\documentclass[12pt]{amsart}

\usepackage{amsmath, amscd, latexsym, hyperref, times, color}
\usepackage{subcaption, xcolor, graphicx,import}

\oddsidemargin 0in \theoremstyle{plain} \topmargin 0in

\oddsidemargin .25in \theoremstyle{plain}

\newtheorem{Thm}{Theorem}
\newtheorem{Lem}[Thm]{Lemma}
\newtheorem{Conj}[Thm]{Conjecture}

\newtheorem{Prop}[Thm]{Proposition}

\theoremstyle{definition}
\newtheorem{Rem}[Thm]{Remark}

\newcommand{\obg}{\mathrm{obg}}
\newcommand{\g}{\mathrm{g}}
\newcommand{\BZ}{\mathbb{Z}}
\newcommand{\id}{\mathrm{id}}
\newcommand{\inv}{^{-1}}

\def\v{\vskip.12in}

\begin{document}

\title{Non-prime 3-Manifolds with Open Book Genus Two}

\author{Mustafa Cengiz}

\address{Department of Mathematics \\ Boston College \\ Chestnut Hill, Massachusetts, USA}

\email{mustafa.cengiz@bc.edu}

\begin{abstract}
An open book decomposition of a 3-manifold $M$ induces a Heegaard splitting for $M$, and the minimal genus among all Heegaard splittings induced by open book decompositions is called the \emph{open book genus} of $M$. It is conjectured by Ozbagci \cite{O} that the open book genus is additive under the connected sum of 3-manifolds.  In this paper, we prove that a non-prime 3-manifold which has open book genus 2 is homeomorphic to $L(p,1)\#L(q,1)$ for some integers $p,q\neq\pm1$, that is, it has non-trivial prime pieces of open book genus 1. In particular, there cannot be a counter-example to additivity of the open book genus such that the connected sum has open book genus 2.
\end{abstract}

\thanks{}

\v \v \v

\maketitle

\setcounter{section}{0}

%--------------------------------------------------------------

\section*{Introduction}
Throughout this paper, any 3-manifold is closed, connected and orientable. A pair $(B,\pi)$ is called an \emph{(embedded) open book decomposition} of a 3-manifold $M$ if $B$ is a link in $M$ and $\pi:(M\setminus B)\to S^1=[0,1]/\sim$ is a fibration with fibers realizing $B$ as boundary. The link $B$ is called the \emph{binding} of the open book, and the closure of each fiber of $\pi$ in $M$ is called a \emph{page} of the open book. Note that the pages have the homeomorphism type of a compact surface $\Sigma$. If $(B,\pi)$ is an open book decomposition of $M$, then $H_1=(\pi\inv([0,1/2])\cup B)$ and $H_2=(\pi\inv([1/2,1])\cup B)$ are handlebodies homeomorphic to $\Sigma\times I$ embedded in $M$. Moreover, $(H_1,H_2)$ defines a Heegaard splitting of genus $g=g(\Sigma\times I)=1-\chi(\Sigma)$ for $M$. The \emph{open book genus} of a $3$-manifold $M$, denoted by $\obg(M)$, is the minimum genus over all Heegaard splittings of $M$ induced by open book decompositions. It immediately follows from the definition that $\g(M)\le \obg(M)$ for any 3-manifold $M$, where $g(M)$ is the Heegaard genus of $M$. A Heegaard splitting is not necessarily induced by an open book decomposition. Furthermore, $\obg(M)$  does not necessarily equal $g(M)$. For example, ``most'' 3-manifolds of Heegaard genus 2 have greater open book genus \cite{R}. We also have the following classification of $3$-manifolds with open book genus 0 and 1 (see \cite{O}).
\begin{Prop}
Let $M$ be a 3-manifold. Then $\obg(M)=0$ if and only if $M\cong S^3$, and $\obg(M)=1$ if and only if $M\cong L(p,1)$ for some integer $p\neq \pm1$.
\end{Prop} 
A well-known corollary of Haken's Lemma is that the Heegaard genus is additive under the connected sum of 3-manifolds. The following conjecture is stated by Ozbagci \cite{O}.
\begin{Conj}\label{conj}
The open book genus is additive under the connected sum of 3-manifolds.
\end{Conj}
For $M=M_1\#M_2$, proving that $\obg(M)\le \obg(M_1)+\obg(M_2)$ is straightforward. For $i=1,2$, take open book decompositions of $M_i$ with pages of homeomorphism type $\Sigma_i$ inducing Heegaard splittings of genus $1-\chi(\Sigma_i)=\obg(M_i)$. A plumbing of these open book decompositions is an open book decomposition for $M=M_1\#M_2$ with pages homeomorphic to $\Sigma=\Sigma_1+\Sigma_2$, which is obtained by gluing closed rectangles $R_i\subset \Sigma_i$ in a certain way (see \cite{E} for details). It follows that $\chi(\Sigma)=\chi(\Sigma_1)+\chi(\Sigma_2)-1$. Since a plumbing of the open books induces a Heegaard splitting of genus $1-\chi(\Sigma)$ for $M$, we get $$\obg(M)\le1-\chi(\Sigma)=1-\chi(\Sigma_1)+1-\chi(\Sigma_2)=\obg(M_1)+\obg(M_2).$$
However, it is still unknown whether the $\obg$ is super-additive, equivalently additive, or not. In this paper, we prove that there cannot be a counter-example to the Conjecture \ref{conj} such that the connected sum has open book genus 2. More precisely, we prove the following.
\begin{Thm}\label{non-prime}
A non-prime 3-manifold $M$ has open book genus 2 if and only if each non-trivial connected summand of $M$ has open book genus 1, that is, $M\cong L(p,1)\#L(q,1)$ for some integers $p,q\neq\pm1$.
\end{Thm}
An open book decomposition induces a genus 2 Heegaard splitting if and only if it has pages of Euler characteristic $-1$. Therefore, the pages are either once-punctured tori or pairs of pants. We will use different tools to deal with each case. In section 1, we recall the fact that a 3-manifold that has an open book decomposition with once-punctured torus pages is a double branced cover of $S^3$ along a 3-braid link, and we list some known facts about double branched covers and braid closures. In section 2, we analyze 3-manifolds which possess open book decompositions with pair of pants pages. We argue that if such a manifold is not prime, then the open book decomposition should be ``simple". Our tool in section 2 is the theory of Seifert fibered spaces. Finally, we prove Theorem \ref{non-prime} in section 3.

Note that the author of \cite{O} refers to the open book genus as the \emph{contact genus} to draw attention to the strong relation between open book decompositions and contact structures, namely the Giroux correspondence. However, we do not benefit from this relation and use classical tools from 3-manifolds topology.

\vspace{0.2in} \noindent{\bf {Acknowledgements.}} The author is grateful to Tao Li for numerous helpful discussions. He also would like to thank Kyle Hayden and Mike Miller for pointing out useful facts from Knot Theory.

\section{Open Book Decompositions of Double Branched Covers}
In this section, we recollect useful facts from classical Knot Theory. Let us denote the double branched cover of $S^3$ along a link $L$ by $M(L)$. See \cite{Rol} for definition and details. The topology of $M(L)$ is strongly related to properties of $L$. It is known that every link $L$ can be represented as a braid closure \cite{A}. The minimum number of strands required to represent $L$ as a braid closure is called the \emph{braid index} of $L$, denoted by $b(L)$. Assume that $L$ is represented as a braid closure along $k$ strands with braid axis $A$. Since $A$ is the unknot, it is the binding of an open book decomposition $(A,p)$ of $S^3$, where $p$ is the projection map $(S^3\setminus A)=\mathrm{int}(D^2)\times S^1\to S^1$. Let $D_t$ be the open disk $p\inv(t)$ and $f:M(L)\to S^3$ the branched covering map. Define $B=f\inv(A)$ and $\pi=p\circ f:(M(L)\setminus B)\to S^1$. It follows that $\pi$ is a fibration with fibers $\Sigma_t=f\inv(D_t)$ realizing $B=f\inv(A)$ as boundary, i.e., $(B,\pi)$ is an open book decomposition of $M(L)$. On the other hand, the braid representation of $L$ on $k$ strands intersects each page $D_t$ in $k$ points, and $\Sigma_t$ is a double branched cover of $D_t$ along these $k$ points. In other words, the pages of $(B,\pi)$ are homeomorphic to a surface $\Sigma$, which is a double branched cover of $D^2$ along $k$ points. It follows that $\chi(\Sigma)=2-k$, and hence the induced open book decomposition $(B,\pi)$ of $M(L)$ induces a Heegaard splitting of genus $1-\chi(\Sigma)=k-1$. This proves the following.
\begin{Prop}\label{bound}
If $L$ is a link in $S^3$ with $b(L)=k$, then $\obg(M(L))\le k-1$.
\end{Prop} 
Since the braid index of a link $L$ suggests an upper bound for the open book genus of $M(L)$, it is worth stating the following result concerning the additivity of the braid index under connected sum.
\begin{Thm}[\cite{BM}, The Braid Index Theorem]\label{additive}
The braid index is minus one additive under the connected sum operation, that is, $b(L_1\#L_2)=b(L_1)+b(L_2)-1$.
\end{Thm}
In the discussion above, since $\Sigma$ is a double branched cover of $D^2$ along $k$ points, it is a compact genus $g$ surface with $b$ boundary components, where
$$(g,b)=\begin{cases}
g=(k-1)/2\text{ and }b=1,\text{ if $k$ is odd,}\\
g=(k-2)/2\text{ and }b=2,\text{ if $k$ is even.}
\end{cases}$$
In particular, if $L$ is a 3-braid (respectively 2-braid) link, then $M(L)$ has an open book decomposition with once-punctured torus (respectively annulus) pages. On the other hand, any 3-manifold $M$ that has an open book with once-punctured torus pages is homeomorphic to $M(L)$ for some 3-braid link $L$ (see \cite{Bald}, section 2). Thus, we have the following.
\begin{Thm}\label{3-braid}
A 3-manifold $M$ has an open book decomposition with once-punctured torus pages if and only if it is homeomorphic to a double branched cover of $S^3$ along a 3-braid link $L$.
\end{Thm}
\begin{Rem} This statement does not generalize for $k$-braid links when $k\ge 4$. In particular, there are open book decompositions with connected bindings which are not obtained as double branched covers along braids. This can be argued more carefully, but we simply note that every 3-manifold has an open book with connected binding, whereas there are 3-manifolds, e.g. $S^1\times S^1\times S^1$, which are not double branched covers of $S^3$ \cite{Rol}.
\end{Rem}
We can say more about the topology of $M(L)$ looking at the link $L$. For example, it is known that $M(L_1\#L_2)\cong M(L_1)\#M(L_2)$ because a sphere $S$ that realizes the connected sum $L_1\#L_2$ in $S^3$ lifts to a sphere in $M(L_1\#L_2)$ that realizes the connected sum $M(L_1)\#M(L_2)$. It is also known that $M(L)\cong S^3$ if and only if $L$ is the unknot \cite{MB}. These two facts imply that $M(L)$ is a non-prime manifold whenever $L$ is a composite link since if $L=L_1\#L_2$ for non-trivial links $L_i$, then $M(L)\cong M(L_1)\#M(L_2)$ for non-trivial 3-manifolds $M(L_i)$. The converse is also true.
\begin{Thm}[\cite{KM}, Corollary 4]\label{prime}
A (non-split) link $L$ is prime if and only if $M(L)$ is prime.
\end{Thm}
\section{Open Book Decompositions with Pair of Pants Pages}
We can define open book decompositions in an abstract way. Let $\Sigma$ be a compact surface with boundary and $\phi$ an orientation-preserving self-homeomorphism of $\Sigma$, called a \emph{monodromy}, such that $\phi|_{\partial\Sigma}$ is the identity map. Take the mapping torus $M_\phi$ and perform \emph{vertical} Dehn fillings on $\partial M_\phi$, i.e., glue a solid torus $S^1\times D^2$ to each torus component of $\partial M_\phi=\partial\Sigma\times S^1$ via a homeomorphism identifying a meridian $\{\ast\}\times S^1$ to $\{p\}\times S^1$ for some $p\in\partial\Sigma$. The resulting space forms a 3-manifold $M$ with an embedded open book decomposition such that the binding is the cores of the glued solid tori, and the pages are homeomorphic to $\Sigma$. The pair $(\Sigma,\phi)$ is called an \emph{abstract open book decomposition} of $M$. Every embedded open book decomposition of a 3-manifold $M$ can be viewed as an abstract open book decompostion as well \cite{E}. We use the notation $M=(\Sigma,\phi)$. Note that isotoping or conjugating $\phi$ does not change the homeomorphism type of $M$.

An abstract open book decomposition $(\Sigma,\phi)$ of a 3-manifold $M$ defines an extrinsic Heegaard splitting for $M$. Namely, we take two copies of the handlebody $\Sigma\times I$ and glue them along their boundaries $\partial (\Sigma\times I)=\Sigma_0\cup (\partial\Sigma\times I)\cup \Sigma_1$ via the homeomorphism $f$ defined by $f|_{\Sigma_0}=\id$, $f|_{\partial\Sigma\times I}=\id$ and $f|_{\Sigma_1}=\phi$, where $\Sigma_t$ denotes $\Sigma\times\{t\}$. This can be seen as follows.
When we glue $\Sigma_0$'s via the $\id$ map, we obtain the trivial interval bundle over $\Sigma$. Then, gluing $\Sigma_1$'s via $\phi$, we obtain the mapping torus $M_\phi$. Finally, vertical Dehn fillings on $\partial M_\phi$ can be recovered by identifying $\partial\Sigma\times I$'s via the identity map. Hence, the open book decomposition and the Heegaard splitting form homeomorphic 3-manifolds.

\begin{figure}[h]
	\centering
	\def\svgwidth{0.3\linewidth}
	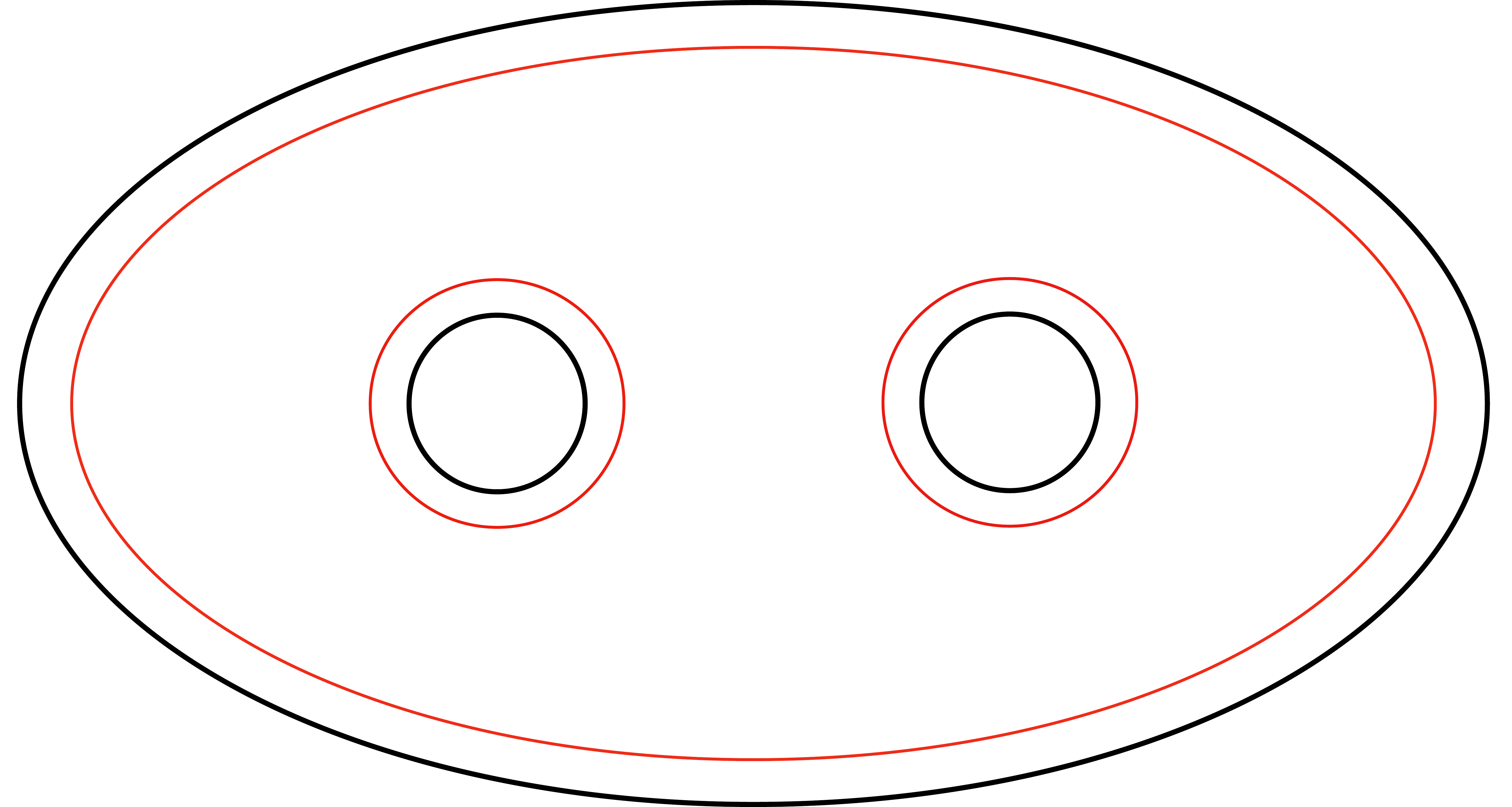
	\caption{Generators of $\mathrm{Mod}(P)$.\label{Pop}}
\end{figure}
Some 3-manifolds with open book genus 2 can be obtained as abstract open book decompositions with pages homeomorphic to a pair of pants $P$. The mapping class group $\mathrm{Mod}(P)$ of $P$ is homomorphic to $\BZ^3$, where the generators are the Dehn twists $t_i$ about the curves $c_i$ in Figure \ref{Pop}. Therefore, any given monodoromy $\phi$ of $P$ can be written as a product of powers of $t_i$ up to isotopy, that is, $\phi$ can be taken to be $t_1^{r_1}t_2^{r_2}t_3^{r_3}$ for some integers $r_i$. In this section, we will argue that possible values of $r_i$ are pretty restricted if the 3-manifold $M=(P,\phi)$ is not prime.
\begin{Lem}\label{SFS}
If no $r_i$ equals 0, then $M=(P,\phi)$ is a Seifert fibered space.
\end{Lem}
Apparently, this lemma is known to experts (see \cite{O}), however, the author could not find a proof in the literature. We present a proof here.
\begin{proof}
Let $d_1,d_2,d_3$ be the boundary components of $P$ parallel to the curves $c_1,c_2,c_3$ in Figure \ref{Pop}. For $i=1,2,3$, let $a_i$ be an embedded curve in $P$ so that $a_i$ and $d_i$ bound an annulus $A_i$ with core $c_i$. Finally, let $b_i$ be a spanning arc for $A_i$, and $p_i$ the endpoint of $b_i$ on $d_i$ as in Figure \ref{Pieces}.

\begin{figure}[h]
	\centering
	\def\svgwidth{0.5\linewidth}
	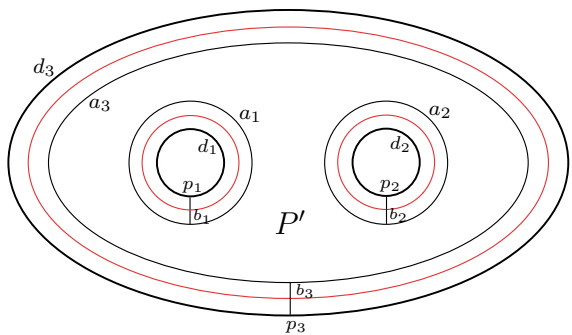
	\caption{The curves $a_i, b_i,c_i,d_i$, and the points $p_i$ on $P$\label{Pieces}.}
\end{figure}

When we remove the annuli $A_i$ from $P$, we obtain another pair of pants $P'\subset P$. Notice that $\phi|_{P'}$ is the identity map, and so $P'\times S^1$ is embedded in the mapping torus $M_\phi$. We will analyze how the vertical Dehn fillings on the boundary components $d_i\times S^1$ of $M_\phi$ are realized on the boundary components $a_i\times S^1$ of $P'\times S^1$. Let $M_i$ be the component of $M_\phi\setminus (P'\times S^1)$ with boundary $(d_i\times S^1)\cup (a_i\times S^1)$. In other words, $M_i$ is the mapping torus of the monodromy $t_i^{r_i}$ on the annulus $A_i$, where $t_i$ is the Dehn twist about $c_i$.\vspace{2mm}

\noindent\textit{Claim.} There is a properly embedded annulus  $B_i$ in $M_i$ such that $B_i\cap (d_i\times S^1)=\{p_i\}\times S^1$ and $B_i\cap (a_i\times S^1)$ is an $(r_i,1)$-curve on the torus $a_i\times S^1$.\vspace{2mm}

\noindent\textit{Proof of the claim.} Take the rectangle $R_i=b_i\times I$ in $A_i\times I$, and let $q_i$ be the vertex $(a_i\cap b_i) \times \{1\}$ of $R_i$. Rotate $q_i$ on $a_i\times\{1\}$ for $r_i$ times in the positive (negative) direction if $r_i$ is positive (negative), while keeping the other vertices of $R_i$ fixed. In Figure \ref{Piecesa}, we show $R_i$ before rotation, and in Figure \ref{Piecesb}, we depict $\partial R_i$ after rotation in the $r_i=2$ case. In particular, $R_i\cap (A_i\times \{1\})$ is the image of $b_i$ under the monodromy $t_i^{r_i}$, and $R_i\cap (A_i\times \{0\})$ is the $b_i$ itself. Since $M_i$ is obtained from $A_i\times I$ by identifying $R_i\cap(A_i\times\{0\})=b_i\times \{0\}$ with $R_i\cap (A_i\times \{1\})=t_i^{r_i}(b_i)\times\{1\}$, then $R_i$ turns into a properly embedded annulus $B_i$ in $M_i$. By construction, the boundary of $B_i$ on $d_i\times S^1$ is $\{p_i\}\times S^1$, and the boundary of $B_i$ on $a_i\times S^1$ is an $(r_i,1)$-curve. This completes the proof of the claim.\vspace{2mm}

\begin{figure}[h!]
	\centering
		\begin{subfigure}[b]{0.5\textwidth}
		\centering
		\def\svgwidth{0.61\linewidth}
		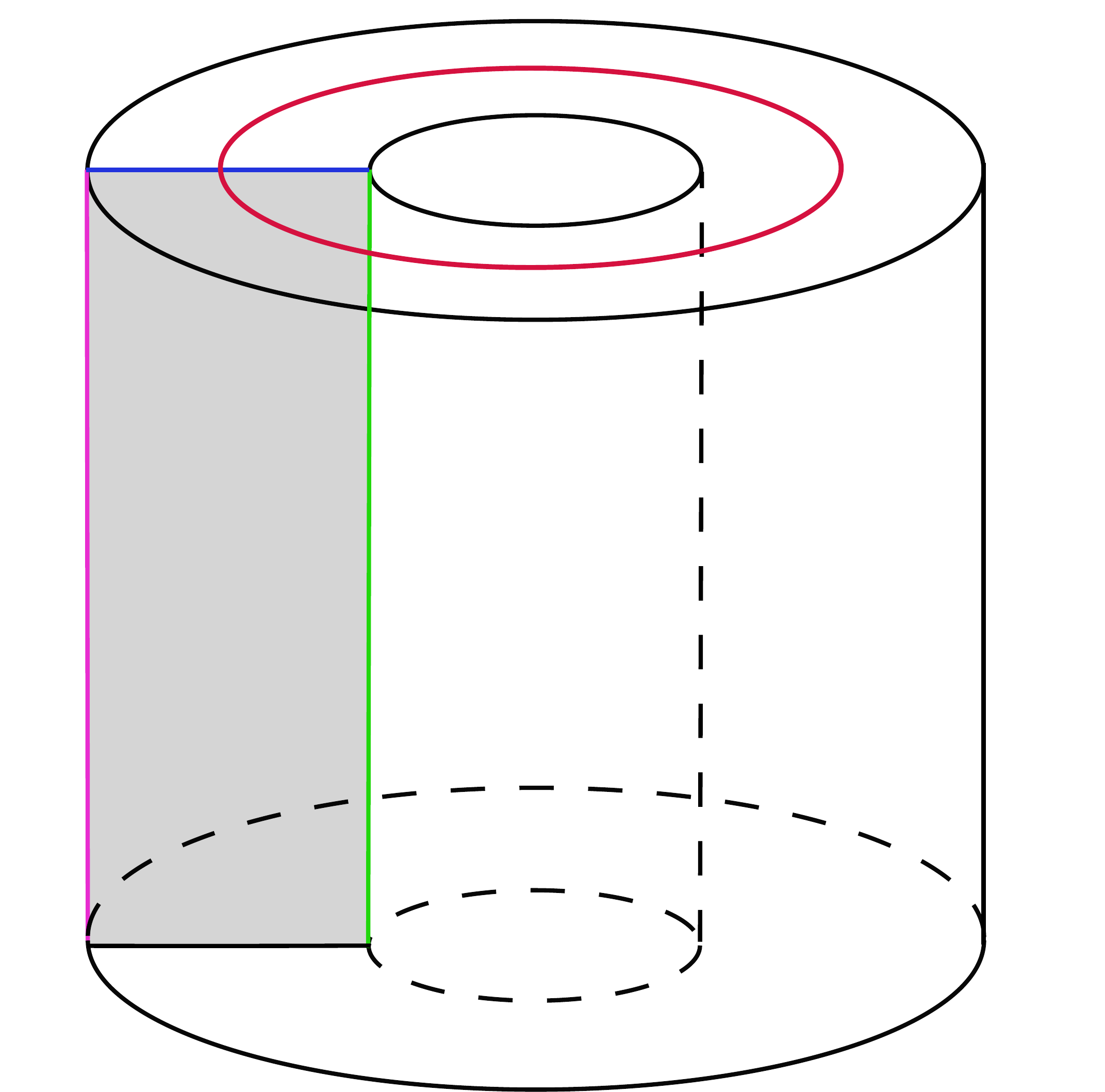
		\caption{$R_i$ before rotation.\label{Piecesa}}
	\end{subfigure}%
	~ 
	\begin{subfigure}[b]{0.5\textwidth}
		\centering
		\def\svgwidth{0.5\linewidth}
		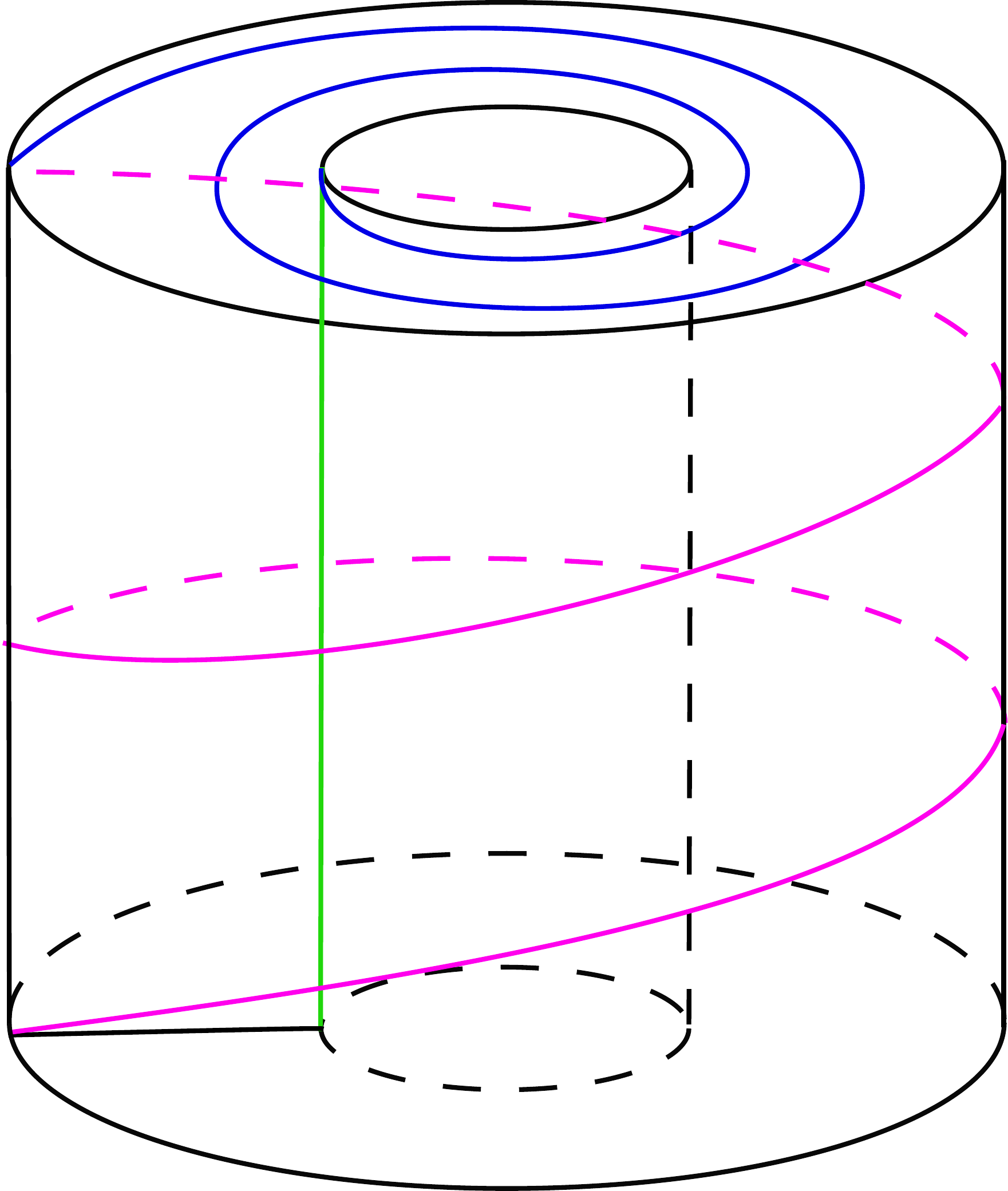
		\caption{$\partial R_i$ after rotation when $r_i=2$.\label{Piecesb}}
	\end{subfigure}
	\caption{Detecting the annulus $B_i$.\label{Annuli}}
\end{figure}

Now, when we perform vertical Dehn fillings on $\partial M_\phi$, we glue a solid torus $S^1\times D^2$ to each boundary component $d_i\times S^1$ by attaching a meridional disk $D=\{\ast\}\times D^2$ of the solid torus to $\{p_i\}\times S^1$. On the boundary component $a_i\times S^1$ of $P'\times S^1$, this is realized as attaching the disk $D\cup_{\{p_i\}\times S^1} B_i$ to the non-infinity slope $(r_i,1)$. Since the result of Dehn fillings along non-infinity slopes $(r_i,1)$ on the components $a_i\times S^1$ of $\partial (P'\times S^1)$ is the Seifert fibered space $M(+0,0;1/r_1,1/r_2,1/r_3)$, the result follows.
\end{proof}
The proof suggests that when no $r_i$ is zero, $M=(P,\phi)$ has a Seifert fibration over the base space $S^2$ with three cone points of multiplicities $|r_i|$, possibly 1. We can prove the following statement using the theory of Seifert fibered spaces.
\begin{Lem}\label{Cor} If $M=(P,\phi)$ is not prime for  $\phi=t_1^{r_1}t_2^{r_2}t_3^{r_3}$, one of the $r_i$'s should be 0.
\end{Lem}
\begin{proof}
Assume that no $r_i$ equals 0, so $M$ is Seifert fibred as above. Let $\pi:M\to S^2$ be the projection map of this fibration. We will prove that $M$ is either irreducible or homeomorphic to $S^2\times S^1$, hence it is prime. If $M$ is reducible, then there exists an essential sphere $S$ in $M$ which intersects each Seifert fiber transversely, and hence $\pi|_S:S\to S^2$ is a branched cover along three cone points with multiplicities $|r_i|$ \cite{H}. Hence, we obtain $\chi(S^2)-\chi(S)/n=\sum_{i=1}^{3}(1-1/|r_i|)$ for $n$ the degree of the cover. As $\chi(S)=\chi(S^2)=2$, we get $|r_i|=1$ for each $i$, and $n=1$. This implies that $S$ intersects each fiber once, and $\overline{M\setminus N(S)}$ is homeomorphic to $S\times I$. In other words, $M$ is an orientable $S^2$ bundle over $S^1$, and thus $M$ is homeomorphic to $S^2\times S^1$.
\end{proof}

\section{Proof of  Theorem \ref{non-prime}}
The if direction of Theorem \ref{non-prime} is straightforward. If $M_1,M_2$ are 3-manifolds with open book genus 1, then they have Heegaard genus 1. Hence, $\g(M)=2$ for $M=M_1\#M_2$. It follows that $2=\g(M)\le\obg(M)\le2$, by the subadditivity of the open book genus. Therefore, we obtain $\obg(M)=2$.

For the only if direction of Theorem \ref{non-prime}, let $M$ be a non-prime 3-manifold with open book genus $2$. Pick an open book decomposition of $M$ inducing a genus 2 Heegaard splitting. The pages of such an open book have Euler characteristic $-1$, hence they are either once-punctured tori or pairs of pants. We analyze each case separately. \vspace{1mm}

\noindent\textbf{Case 1.} $M$ has an open book decomposition with once-punctured torus pages.\vspace{1mm}

By Theorem \ref{3-braid}, $M$ is homeomorphic to a double branched cover of $S^3$ along a 3-braid link $L$. Theorem \ref{prime} implies that $L$ is a composite link, that is, $L=L_1\#L_2$ for some non-trivial links $L_i$, since $M$ is not prime.  On the other hand, by Proposition \ref{bound}, $b(L)$ must be 3 since otherwise $\obg(M(L))$ would be less than 2. Finally, it follows from Theorem \ref{additive} that $b(L)=3=b(L_1)+b(L_2)-1$ for braid indices, so $b(L_1)=b(L_2)=2$. Therefore, we get $M\cong M(L_1\#L_2)\cong M(L_1)\#M(L_2)$,
which implies that the prime pieces $M(L_1),M(L_2)$ of $M$ are double branched covers of $S^3$ along 2-braid links. The coverings suggest open book decompositions with annulus pages inducing genus 1 Heegaard splittings for $M(L_1)$ and $M(L_2)$. Hence, the result follows. \vspace{1mm}

\noindent\textbf{Case 2.} $M$ has an open book decomposition with pair of pants pages.\vspace{1mm}

We can see $M$ as an abstract open book $(P,\phi)$ for $P$ the pair of pants with monodromy $\phi=t_1^{r_1}\,t_2^{r_2}\,t_3^{r_3}$, where $t_i$ are the Dehn twists about the curves $c_i$ in Figure \ref{Pop}. By Lemma \ref{Cor}, one of the $r_i$'s is zero because $M$ is not prime. Assume that $r_3=0$, so $\phi=t_1^{r_1}\,t_2^{r_2}$. We now analyze the Heegaard splitting induced by $(P,\phi)$.

\begin{figure}[h]
	\centering
	\def\svgwidth{0.7\linewidth}
	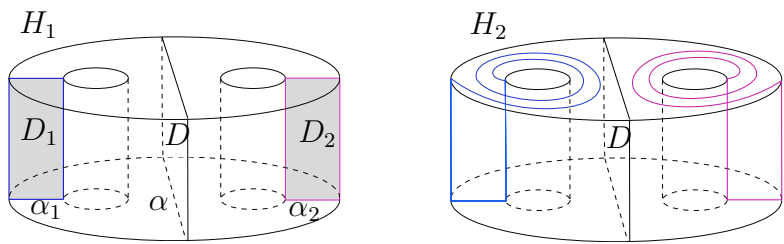
	\caption{The induced Heegaard splitting when $r_1=2$ and $r_2=-2$ is homeomorphic to $L(2,1)\#L(-2,1)$.\label{Splitting}}
\end{figure}

Take two distinct copies $H_1, H_2$ of $P\times I$. Let $\alpha\subset P$ be a properly embedded essential arc with endpoints in $d_3$ such that $\alpha$ intersects neither $c_1$ nor $c_2$, hence $\phi$ fixes $\alpha$ pointwise. Let $\alpha_1,\alpha_2\subset P$ be properly embedded non-separating arcs which do not intersect $\alpha$ and cut $P$ into a disk. Now take the vertical disks $D=\alpha\times I$, $D_1=\alpha_1\times I$ and $D_2=\alpha_2\times I$ in $H_1$ as in Figure \ref{Splitting}. The disks $D_1, D_2$ cut $H_1$ into a 3-ball, and hence	the resulting 3-manifold is uniquely determined by where $\partial D_1,\partial D_2$ are mapped in $\partial H_2$ under the Heegaard map $f: \partial H_1 \to \partial H_2$ defined by $f|_{P_0}=\id$, $f|_{\partial P\times I}=\id$ and $f|_{P_1}=\phi$, where $P_t=P\times\{t\}$. In Figure \ref{Splitting}, we depict $f(\partial D_1),f(\partial D_2)\subset \partial H_2$ in the case $r_1=2$ and $r_2=-2$. Figure \ref{Splitting} is suggestive, and in general, two distinct copies of $D$ glue along their boundaries to form a sphere $S\subset M$ that realizes the connected sum $L(r_1,1)\#L(r_2,1)$. Finally, note that $r_1,r_2\neq \pm 1$. Otherwise, one of $L(r_i,1)$ would be $S^3$, and $M$ would be homeomorphic to either $S^3$, $S^1\times S^2$, or a lens space. Thus, the connected summands of $M$ have open book genus 1. \qed

\end{document}